\documentclass[submission,copyright,creativecommons]{eptcs}
\providecommand{\event}{GASCOM 2024} % Name of the event you are submitting to

\usepackage{iftex}

\ifpdf
  \usepackage{underscore}         % Only needed if you use pdflatex.
  \usepackage[T1]{fontenc}        % Recommended with pdflatex
\else
  \usepackage{breakurl}           % Not needed if you use pdflatex only.
\fi
\usepackage{amsmath}
\usepackage{amsfonts}
\usepackage{amssymb}
\usepackage{amsthm}
\usepackage[capitalise]{cleveref}
\usepackage{multicol}

\newtheorem{theorem}{Theorem}

\newtheorem{lemma}[theorem]{Lemma}
\newtheorem{proposition}[theorem]{Proposition}

\newtheorem{conjecture}[theorem]{Conjecture}

\newcommand\pcl{\mathcal{P}}
\DeclareMathOperator{\pal}{Pal}
\DeclareMathOperator{\bw}{bw}
\DeclareMathOperator{\fact}{Fact}
\DeclareMathOperator{\alp}{Alph}

\title{Perfectly Clustering Words and Iterated Palindromes \\ over a Ternary Alphabet}
\author{Mélodie Lapointe \qquad\qquad Nathan Plourde-Hébert
\institute{Université de Monction\\
Moncton, Canada}
\email{melodie.lapointe@umoncton.ca \qquad enp0579@umoncton.ca}
}
\def\titlerunning{Perfectly Clustering Words and Iterated Palindromes}
\def\authorrunning{M. Lapointe \& N. Plourde Hébert}
\begin{document}
\maketitle

\begin{abstract}
    Recently, a new characterization of Lyndon words that are also perfectly clustering was proposed by Lapointe and Reutenauer (2024). A word over a ternary alphabet $\{a,b,c\}$ is called perfectly clustering Lyndon if and only if it is the product of two palindromes and it can be written as $a\pi_1  b  \pi_2 c$ where $\pi_1$ and $\pi_2$ are palindromes. We study the properties of palindromes appearing as factors $\pi_1$ and $\pi_2$ and their links with iterated palindromes over a ternary alphabet.
\end{abstract}

\section{Introduction}

The Burrows-Wheeler transform of a word $w$, denoted $\bw(w)$,  is obtained from $w$ by first listing the conjugates of $w$ in lexicographic order, then concatenating the final letters of the conjugates in this order.
For example, the Burrows-Wheeler transform of $apartment$ is $tpmteaanr$. 
It was introduced in~\cite{BW1994}  as a tool in data compression. After applying the Burrows-Wheeler transform to a word, the occurrence of a given letter tend to occur in clusters. This clustering effect is optimal when all occurrences of each letter are group together. Words showing that optimal properties are thus called $\pi$-clustering. The permutation $\pi$ represent the order in which the cluster of similar letters appear. 
The word aluminium, for example, is $451623$-clustering since $bw(aluminium) = mmnauuiil$.
A word $w$ is \emph{perfectly clustering} if its Burrows-Wheeler transform is a decreasing word, i.e., the clusters of letters appear from highest to lowest with respect to the alphabet order. 
This terminology was introduced  by Ferenczi and Zamboni~\cite{FZ2013}.

Perfectly clustering words were proposed in~\cite{PS2008, FZ2013} as a generalization of Christoffel words.
Recently, Reutenauer and the first author~\cite{LR2024} showed that a primitive word $w$ is a perfectly clustering Lyndon word if and only if it is a product of two palindromes and has a palindromic special factorization, i.e., $w = a_1\pi_1 a_2 \pi_2 \cdots \pi_{k-1}a_{k}$, where the letters in $w$ are in $\{a_1 < a_2 < \dots < a_k\}$ and $\pi_1, \pi_2,\dots, \pi_k$ are palindromes. This is also a generalization of characterization of Christoffel word due to de Luca and Mignosi~\cite{dLM1994}; a binary word $amb$ is a Christoffel word if and only if the word $amb$ is a product of two palindromes and $m$ is also a palindrome called a central word. Hence, the palindromic special factorization of a Christoffel word in $\{a,b\}^*$ is simply $amb$ where $m$ is a palindrome. Central words have many properties (see~\cite{B2007,R2019} for more information). We recall only one of them: a central word is the image of a mapping called iterated palindromization~\cite{dL1997}. 

In this extended abstract, we discuss results about the palindromes in the palindromic special factorization of perfectly clustering Lyndon words over ternary alphabet. In Section 2, we recall some definitions about these words. In Section 3, 
we explore some relationships between the palindromes appearing in the special factorization.
In Section 4, we describe the iterated palindromes that are factors of this factorization of perfectly clustering Lyndon words.

\section{Definition}

\subsection{Words}
For the rest of the paper, 
let $A = \{a,b,c\}$ be a totally ordered alphabet, where $a < b < c$.
Let $w = w_1w_2\cdots w_n$ be a word in the free monoid generated by $A$.
The length of $w = w_1\cdots w_n$ (with $w_i\in A$), denoted by $|w|$, is $n$.
The number of occurrences of a letter $x$ in $w$ is denoted by $|w|_x$.
The {\em Parikh vector} of $w$ is the integer vector $(|w|_{a},|w|_b,|w|_{c})$.
The function $\alp$ is defined by $\alp(w) = \{x \in A \mid |w|_x \geq 1\}$. 

A word $w$ is called {\em primitive} if it is not the power of another word; that is, for any word $z$ such that $w = z^n$, one has $n = 1$.
The \emph{conjugates} of a word $w$ are the words $w_i \cdots w_n w_1 \cdots w_{i-1}$. In other words, two 
words $u,v\in A^*$ are {\it conjugate} if for some words $x,y\in A^*$, one has $u=xy$ and $v=yx$. The {\em conjugation class} of a word is the set of its conjugates.
If a word $w$ of length $n$ is primitive, then it has  exactly $n$ distinct conjugates.
A word $w$ is called a {\em Lyndon word} if it is  primitive, and 
it is the minimal word in lexicographic order among its conjugates.

The {\it reversal} of $w = w_1\cdots w_n$, denoted by $R(w)$, is the word $R(w) = w_n \cdots w_1$. 
A {\it palindrome} is a word $w$ such that $w = R(w)$.
A word $u$ is a {\it factor} of $w$ if there exists two words $x,y \in A^*$ such that $w=xuy$.
The {set of factors} of $w$ is denoted by $\fact(w)$ and $\fact_k(w)$ denotes the set of factors of length $k$ of $w$. 

\subsection{Perfectly Clustering Lyndon Words}\label{pcw}
The {\it special factorization} of a word $w$ over $A$ is a factorization of $w$ of the form
$w= a\pi_1 b  \pi_2 c$, where $\pi_1, \pi_2 \in A^*$. If $\pi_1$ and $\pi_2$ are both palindromes, then the special factorization is called {\it palindromic}.
A {\it perfectly clustering Lyndon words} on $A^*$ is a word $w$ such that $w$ is a product of two palindromes and has a palindromic special factorization. 
For example, the word $acbcbbcbc$ is a perfectly clustering Lyndon word since it is the product of the palindromes $a$ and $cbcbbcbc$ and it has the palindromic special factorization $a\cdot cbc \cdot b \cdot bcb \cdot c$.
Moreover, the palindromic special factorization of a perfectly clustering Lyndon word is unique~\cite{LR2024}.

This is not the original definition of perfectly clustering words, but of a characterization of perfectly clustering Lyndon word given in~\cite{LR2024}. Usually, a word $w$ is called perfectly clustering if its Burrows-Wheeler transform is $c^{|w|_c}b^{|w|_b}a^{|w|_a}$ (see~\cite{PS2008} for a complete definition).
If a primitive word is perfectly clustering, then all its conjugates are. Consequently, there is no loss of generality in studying only perfectly clustering Lyndon words.
The set of perfectly clustering Lyndon words is denoted by $\pcl$.

It was proved by Mantaci, Restivo and Sciortino~\cite[Theorem 9]{MRS2003} that perfectly clustering words on a binary alphabet are  Christoffel words 
and their conjugates.
Let recall the following lemma describing the possible sets of factors of length 2 of a perfectly clustering word.
\begin{lemma}[\cite{PS2008}]~\label{2-factors}
    Let $w$ be a perfectly clustering word in $\{a,b,c\}^*$. Then $\fact_2(w)$ is a subset of one of the sets below:
    \begin{multicols}{3}
    \begin{itemize}
        \item $\{ab,ac,ba,bb,ca\}$
        \item $\{aa,ab,ac,ba,ca\}$
        \item $\{ac,bb,bc,ca,cb\}$
        \item $\{ac,bc,ca,cb,cc\}$
    \end{itemize}
    \end{multicols}
\end{lemma}

\subsection{Iterated Palindromes}
The \emph{(right) palindromic closure} of $w$, denoted by $w^{(+)}$, is the shortest unique palindrome having $w$ as a prefix, i.e., if $w = ps$ where $s$ is the longest palindromic suffix of $w$, then $w^{(+)} = ps R(p)$. We define the mapping $\pal(w)$ from a free monoid to itself, called iterated palindromization, as follows: $\pal(\varepsilon) = \varepsilon$ and for each letter $x$, $\pal(ux) = (\pal(u)x)^{(+)}$. A word $w$ such that $w = \pal(u)$ is called an \emph{iterated palindrome} and $u$ is called the \emph{directive word} of $w$.
For example, the word $ababaababa$ is an iterated palindrome and its directive word is $abba$.
A word $amb$ is a Christoffel word if and only if $m$ is an iterated palindrome on a binary alphabet \cite{dL1997}. Therefore, iterated palindromes on a binary alphabet are the only palindromes in the palindromic special factorization of Christoffel words. 

\section{Sets of palindromes}

From the set $\pcl$, let define two sets of words $P_1$ and $P_2$ as follows.
\begin{align*}
    P_1 &= \{\pi_1 \mid a\pi_1b\pi_2 c \in \pcl\} \\
    P_2 &= \{\pi_2 \mid a\pi_1b\pi_2 c \in \pcl\}. 
\end{align*}
By definition, we know that all the words in $P_1$ and $P_2$ are palindromes. However, these sets are not equal, as shown in the next proposition.

\begin{proposition}
    $P_1 \neq P_2$
\end{proposition}

\begin{proof}
    One can check that $a\cdot cbc \cdot b \cdot bcb \cdot c$ is a perfectly clustering Lyndon word with the given palindromic special factorization. Hence, the palindrome $bcb \in P_2$. 
    
    It is sufficient to show that $bcb \not\in P_1$, i.e., that for any word $u \in A^*$, the word
    $a\cdot bcb \cdot b \cdot u \cdot c$ is not a perfectly clustering Lyndon word. The set of factors $\{ab,bc,cb,bb\} \subseteq \fact_2(abcbbuc)$ but $\{ab,bc,cb,bb\}$ is not a subset of the one of the set in Lemma\ref{2-factors}. Thus, the word $a\cdot bcb \cdot b\cdot u \cdot c$ is not a perfectly clustering Lyndon word and $bcb \not\in P_1$. This means that $P_1 \neq P_2$. 
\end{proof}

Some palindromes are in both sets. For example, the words $a\cdot cac \cdot b \cdot c$ and $a\cdot c\cdot b\cdot cac \cdot c$ are both perfectly clustering Lyndon words. Thus, $cac \in P_1 \cap P_2$. The intersection between $P_1$ and $P_2$ is discussed in Section~\ref{test}. 

There is a relationship between $P_1$ and $P_2$. Let $\theta$ be the morphism exchanging the letter $a$ and $c$ defined by
    $\theta(a) = c, \quad \theta(b) = b, \quad \theta(c) = a.$
The antimorphism, $\omega$ defined as  $ \omega = R \circ \theta$, send perfectly clustering Lyndon word to perfectly clustering Lyndon word~\cite{L2020}.

\begin{lemma}~\label{theta}
    $P_1 = \theta(P_2)$
\end{lemma}

\begin{proof}
    Let $p \in P_1$ be an arbitrary palindrome. There exist a perfectly clustering Lyndon word $w$ and a palindrome $u \in A^*$ such that $w = apbuc$. Then
    \begin{align*}
\omega(w) = (R\circ \theta)(apbuc) = R(c\theta(p)b\theta(u)a) = a (R\circ \theta)(u)b (R \circ \theta)(p)c.
    \end{align*}
    The word $\omega(w)$ is a perfectly clustering Lyndon word with the given palindromic special factorization. 
    Since $p$ is a palindrome, $(R\circ \theta)(p) = \theta(p)$ and $\theta(p) \in P_2$. Similarly, we show that $\theta(P_2) \subseteq P_1$. Therefore $P_1 = \theta(P_2)$.
\end{proof}

\section{Iterated palindromes in the previous sets}\label{test}

Some iterated palindromes appears in $P_1$ and $P_2$, but those sets also contain words which are not iterated palindromes. For example, $bacab$ is a palindrome in $P_1$ which is not an iterated palindrome since the word $a\cdot bacab\cdot b\cdot a \cdot c$ is a perfectly clustering Lyndon word.

\begin{proposition}~\label{P1}
    Let $u\in A^*$ be a word. The iterated palindrome $\pal(u)\in P_1$ if and only if $u \in \{a,c\}^*\cdot \{a,b\}^*$.
\end{proposition}

The proof of~\cref{P1} uses induction and the construction of perfectly clustering Lyndon word proposed in~\cite{L2020}.
We defined four automorphisms of the free group $F(A)$ by $\lambda_a (a) = a$, $\lambda_a(b) = ab$ and $\lambda_a(c)= ac$; $\lambda_b (a) = ab^{-1}$, $\lambda_b(b) = b$ and $\lambda_b(c)= bc$; $\rho_b (a) = ab$, $\rho_b(b) = b$ and $\rho_b(c)= b^{-1}c$ and $\rho_c (a) = ac$, $\rho_c(b) = bc$ and $\rho_c(c)= c$. It was proved in~\cite{L2020} that for each perfectly clustering word $w \in A^*$ of length at least 3, there exists a shorter perfectly clustering word $u \in A^*$ and an automorphism $f \in \{\lambda_a,\lambda_b,\rho_b,\rho_c\}$ such that $w = f(u)$. Since the word $abac$ is a perfectly clustering Lyndon word, we only need to show that $f(a\varepsilon bac) = a \pal(u)b q c$ where $f = f_{x_1}\circ f_{x_2} \circ \dots \circ f_{x_n}$, $f_{x_i} \in \{\lambda_a,\rho_b,\rho_c\}$, $u = x_1 x_2 \cdots x_n$ and $q \in P_2$.
Moreover, the following lemma means that no other iterated palindrome can be in $P_1$.

\begin{lemma}
    Let $u \in A^*$ be a word such that $\alp(u) = A$. The iterated palindrome $\pal(bu)$ is not in $P_1$, nor in $P_2$.
\end{lemma}

\begin{proof}
    Let $xv\in A^*$ be a word.
    In~\cite{DJP2001}, it is shown that the first letter of a directive word is separating for $\pal(xv)$, i.e., the letter $x$ appears in each factor of length $2$ of $\pal(xv)$. Hence, the letter $b$ appears in each factor of length $2$ of $\pal(xv)$ and $ac \not\in \fact_2(\pal(bu))$. However, $ac$ is a factor in each set given in~\cref{2-factors}. Thus, $\fact_2(\pal(bu))$ cannot be a factor of a perfectly clustering word.
\end{proof}

Using \cref{theta} and \cref{P1}, one may describe the iterated palindromes which are elements of $P_2$.

\begin{proposition}~\label{P2}
    Let $u\in A^*$ be a word. The iterated palindrome $\pal(u)\in P_2$ if and only if $u \in \{a,c\}^*\cdot \{b,c\}^*$.
\end{proposition}

From the previous proposition one may deduce which iterated palindrome ares in $P_1 \cap P_2$.
\begin{proposition}
    Let $u\in A^*$ be a word. The iterated palindrome $\pal(u)\in P_1\cap P_2$ if and only if $u \in \{a,c\}^*\cdot b^*$.
\end{proposition}

Following computer exploration, we believe that the conjecture below is valid.
\begin{conjecture}
    A word $w \in P_1 \cap P_2$ if and only if $\pal(u) = w$ and $u \in \{a,c\}^*\cdot b^*$.
\end{conjecture}

Iterated palindromes represent a small subset of the palindromes in $P_1$ and $P_2$. Those results are a step in the characterization of these sets that the authors intend to pursue. A more general questions is to characterize the palindromes in the special factorization of perfectly clustering Lyndon words on any alphabets.

%\nocite{*}
\bibliographystyle{eptcs}
\bibliography{mybiblio}
\end{document}